\pdfoutput=1
\documentclass[11pt,appendix=true]{scrartcl}

\usepackage{bbm} % optional
\usepackage{amsmath}
\usepackage{amssymb}
\usepackage{mathtools}
\usepackage[]{algorithm}
\usepackage{algpseudocode}
\usepackage{enumitem}

\usepackage{graphicx}
\usepackage{grffile}
\graphicspath{ {Figures/} }
\usepackage{subcaption}
\makeatletter
\algnewcommand\Input{\item[\textbf{Input:}]}%
\algnewcommand\Output{\item[\textbf{Output:}]}%
\algnewcommand\Init{\item[\textbf{Init:}]}%
\makeatother

\usepackage{xcolor}
\input{tudcolours.def}

\usepackage{amsthm}
\newtheorem{myTheorem}{Theorem}

\newtheorem{myLemma}{Lemma}

\newtheorem{myExample}{Example}

\newcommand{\setOfReals}{\mathbb{R}}
\newcommand{\setOfNaturals}{\mathbb{N}}
\newcommand{\setOfNonnegativeIntegers}{\mathbb{N}_0}
\newcommand{\setOfPositiveReals}{\setOfReals_{+}}
\newcommand{\setN}[1]{ [ #1]}
\newcommand{\borel}[1]{\mathcal{B} (#1 )}
\newcommand{\effectiveDomain}[1]{\mathcal{D}(#1)}

\newcommand{\differential}[1]{\, \mathrm{d} #1}

\newcommand{\indicator}[1]{\mathbbm{1}(#1)}

\newcommand{\prob}{\mathsf{P}}
\newcommand{\probOf}[1]{\prob(#1)}

\newcommand{\eqstop}{.}
\newcommand{\eqcomma}{,}
\newcommand{\defeq}{\coloneqq}

\newcommand{\ie}{\textit{i.e.}}

\usepackage{todonotes}

\newboolean{draftversion}
\setboolean{draftversion}{true} %set false for removing comments and watermark

\newboolean{longVersion}
\setboolean{longVersion}{true}  %set false for the shorter version

\newcommand{\proofLocation}[2]{\ifthenelse{\boolean{longVersion}}{\text{#1}}{\text{#2}}}

\newcommand{\tdWasiur}[2][]{\ifthenelse{\boolean{draftversion}}{\todo[inline, color=blue!20, caption={2do}, #1]{\begin{minipage}{\textwidth-4pt}\emph{Remark Wasiur:}\\#2\end{minipage}}}{}}

\newcommand{\tdArnab}[2][]{\ifthenelse{\boolean{draftversion}}{\todo[inline, color=orange!20, caption={2do}, #1]{\begin{minipage}{\textwidth-4pt}\emph{Remark Arnab:}\\#2\end{minipage}}}{}}

\newcommand{\tdGeneral}[2][]{\ifthenelse{\boolean{draftversion}}{\todo[inline, color=green!20, caption={2do}, #1]{\begin{minipage}{\textwidth-4pt}\emph{ToDO:}\\#2\end{minipage}}}{}}

\usepackage{tikz}
\usetikzlibrary{positioning,fit,shapes,backgrounds,circuits.logic.US}

\tikzstyle{server}=[circle, line width=0.5pt, rounded corners=0.1mm, draw=black!100, fill=tud3a!100]
\tikzstyle{vertex}=[circle, line width=0.5pt, draw=black!100, fill=tud0a!50]
\tikzstyle{dispatcher} =[and gate US, line width=0.5pt, draw=black!100, fill=tud1a!100]
\tikzstyle{dotbox} = [draw=white, fill=white, rectangle,  inner sep=10pt, inner ysep=20pt]

\tikzset{three_sided/.style={
		draw=none,rectangle, %fill=tud1d!5,
		append after command={
			[shorten <= -0.5\pgflinewidth]
			([shift={(-1.5\pgflinewidth,-0.5\pgflinewidth)}]\tikzlastnode.north west)
			edge([shift={( 0.5\pgflinewidth,-0.5\pgflinewidth)}]\tikzlastnode.north east)
			([shift={( 0.5\pgflinewidth,-0.5\pgflinewidth)}]\tikzlastnode.north east)
			edge([shift={( 0.5\pgflinewidth,+0.5\pgflinewidth)}]\tikzlastnode.south east)
			([shift={( 0.5\pgflinewidth,+0.5\pgflinewidth)}]\tikzlastnode.south east)
			edge([shift={(-1.0\pgflinewidth,+0.5\pgflinewidth)}]\tikzlastnode.south west)
		}
	}
}

\PassOptionsToPackage{printonlyused,smaller}{acronym}
\usepackage{acronym}

\usepackage{hyperref}

\begin{document}

% \tableofcontents
%\begin{frontmatter}
  \title{Bounds on the spectral radius of real-valued non-negative  Kernels on measurable spaces}
\author{	Wasiur~R.~KhudaBukhsh\footnote{Department of Electrical Engineering and Information Technology,
		Technische Universit\"{a}t Darmstadt, Germany,
		Email: \href{mailto:wasiur.khudabukhsh@bcs.tu-darmstadt.de}{wasiur.khudabukhsh@bcs.tu-darmstadt.de}, 
	ORCID profile: \href{https://orcid.org/0000-0003-1803-0470}{https://orcid.org/0000-0003-1803-0470}  }      \hspace{1.5mm}\href{https://orcid.org/0000-0003-1803-0470}{\includegraphics[width=3mm]{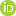}}, \emph{\small Technische Universit\"{a}t Darmstadt} \hfill 	 \\
  Mark Sinzger\footnote{Department of Electrical Engineering and Information Technology,
		Technische Universit\"{a}t Darmstadt, Germany}, \emph{\small Technische Universit\"{a}t Darmstadt}\hfill \\
Heinz Koeppl\footnote{Department of Electrical Engineering and Information Technology,
		Technische Universit\"{a}t Darmstadt, Germany,
		Email:\href{mailto:heinz.koeppl@bcs.tu-darmstadt.de}{heinz.koeppl@bcs.tu-darmstadt.de} }, \emph{\small Technische Universit\"{a}t Darmstadt}  \hfill
}

%\ead{wasiur.khudabukhsh@bcs.tu-darmstadt.de, heinz.koeppl@bcs.tu-darmstadt.de}

%\ead[url]{http://www.bcs.tu-darmstadt.de}

%\address{Bioinspired Communication Systems Lab, \\Department of Electrical Engineering and Information Technology, \\Technische Universit\"at Darmstadt, \\Rundeturmstrasse 12, \\64283 Darmstadt, Germany}


  \date{}

   \maketitle
  \begin{abstract}
  In this short technical note, we extend a recently published result  \cite{LIAO2017Perron} on the Perron root (or the spectral radius) of non-negative matrices to real-valued non-negative kernels on an arbitrary measurable space $(\mathrm{E}, \mathcal{E})$.  To be precise, for any real-valued non-negative kernel $K : \mathrm{E}\times \mathcal{E} \rightarrow \setOfPositiveReals$, we prove that the spectral radius $\rho(K)$ of $K$ satisfies
\begin{align*}
  \inf_{x \in \mathrm{E}  } \frac{ \mathcal{R} K \cdotp L (x) }{  \mathcal{R} L (x) } \le \rho(K) \le   \sup_{x \in \mathrm{E}  } \frac{ \mathcal{R} K\cdotp L (x) }{  \mathcal{R} L (x) } \eqcomma
\end{align*}
where $L$ is an arbitrary Kernel on $(\mathrm{E}, \mathcal{E})$, which is integrable with respect to the left eigenmeasure of $K$ and  satisfies $ \mathcal{R} L (x) >0 $ for all $x \in \mathrm{E}$,   and the operator $\mathcal{R}$ is defined by $\mathcal{R}L (x) \defeq   \int_{\mathrm{E}} L(x, \differential{y}) $.

  \end{abstract}

%  \begin{keyword}



%  \end{keyword}

%\end{frontmatter}

\setcounter{equation}{0}
\section{Introduction}
\label{perron:sec:intro}
Perron roots or spectral radii of matrices and their infinite-dimensional counterpart, non-negative Kernels are a useful quantity in applied probability literature. For instance, the asymptotic behaviour of a \ac{GBP} can be studied via the  eigenvalues and eigenfunctions of the ``expectation'' operators, defined as the partial derivative at zero of the \ac{MGF} of a point-distribution representing the objects in the $1$st generation of the \ac{GBP} (see \cite[Chapter III, p.~59, p.~65]{Harris1963Branching}). A second application concerns \acp{LDP} of waiting times in queueing systems under exogenous Markov modulation. Interestingly, the \ac{LDP} rate function is often found to be linear for many queueing systems. In fact, probability bounds of the following type, which are useful for performance evaluation purposes, are often derived using martingale techniques (and also from \acp{LDP})
\begin{align*}
\lim\limits_{k \rightarrow \infty} \probOf{ W_k \geq \sigma } \leq a(\sigma) e^{ - \theta \sigma  } \eqcomma 
\end{align*}
where $W_k$ is the waiting time for the $k$-th job,   the function $a$ is called the prefactor, and the scalar $\theta$ is called the \emph{effective decay rate} of the queueing system. Here, the effective decay rate~$\theta$ turns out to be the spectral radius of a certain exponentially transformed transition kernel for Markov-modulated systems (see \cite{Rizk2015Sigmetrics,KhudaBukhsh2016GenPE,KhudaBukhsh2018PhD,KhudaBukhsh2018TOMPECS}). Since the effective decay rate~$\theta$  depends on the particular scheduling algorithm chosen, bounds of the above form are very useful for performance evaluation purposes. In essence, the spectral radius $\theta$ determines the performance of the entire system. 

It is evident from the two examples discussed in the previous paragraph that studying  the spectral radius is often beneficial for asymptotic analysis of many stochastic processes arising in applied probability literature. While the exact value of the spectral radius is desirable and most informative, it might be computationally expensive, or even infeasible to find with high degree of accuracy. Moreover, a bound (upper or lower) often suffices for a qualitative analysis. In this short technical note, we extend a recently published result  \cite{LIAO2017Perron} on the Perron root of non-negative matrices to real-valued non-negative kernels on an arbitrary measurable space $(\mathrm{E}, \mathcal{E})$.  To be precise, for any real-valued non-negative kernel $K : \mathrm{E}\times \mathcal{E} \rightarrow \setOfPositiveReals$, we prove that the spectral radius $\rho(K)$ of the kernel $K$ satisfies 
\begin{align*}
\inf_{x \in \mathrm{E}  } \frac{ \mathcal{R} K \cdotp L (x) }{  \mathcal{R} L (x) } \le \rho(K) \le   \sup_{x \in \mathrm{E}  } \frac{ \mathcal{R} K\cdotp L (x) }{  \mathcal{R} L (x) } \eqcomma
\end{align*}
where $L$ is an arbitrary Kernel on $(\mathrm{E}, \mathcal{E})$, which is integrable with respect to the left eigenmeasure of $K$ and  satisfies $ \mathcal{R} L (x) >0 $ for all $x \in \mathrm{E}$,  and the operator $\mathcal{R}$ is defined by $\mathcal{R}L (x) \defeq   \int_{\mathrm{E}} L(x, \differential{y}) $.  In the next section, we provide a short proof of this bound. In \cite{ANSELONE1974Spectral}, the authors proved similar bounds in the context of integral operators with non-negative kernels. 

\setcounter{equation}{0}

\section{Main result}
\label{sec:main_result}
% \citet{LIAO2017Perron} \citet{Harris1963Branching}

\subsection{Notational conventions}
The following notational conventions are adhered to throughout this technical report. We denote the set of natural numbers and the set of real numbers  by $\setOfNaturals$   and $\setOfReals$ respectively. Let  $\setOfNonnegativeIntegers \defeq \setOfNaturals \cup \{0\}$. For $N \in \setOfNaturals$, let $\setN{N} \defeq \{1,2,\ldots, N \}$. The set of non-negative real numbers is denoted by $\setOfPositiveReals$.  For $A \subseteq \setOfReals$, we denote the Borel $\sigma$-field of subsets of $A$ by $\borel{A}$. For any $f: \setOfReals \rightarrow \setOfReals$, we denote the effective domain of $f$ by $\effectiveDomain{f}$, \ie, $\effectiveDomain{f} \defeq \{ x \in \setOfReals \mid \lvert f(x) \rvert < \infty \}$. For an event $A$, we denote the indicator function of $A$ by $\indicator{A}$, taking value unity when $A$ is true and zero otherwise.

\subsection{Bounds on the Perron root}
Let $(\mathrm{E}, \mathcal{E})$ be a given measurable space. The set $\mathrm{E}$ need not be countable, but we  assume $\mathrm{E}$ is  a complete and separable metric space.   Denote the space of real-valued non-negative Kernels on $(\mathrm{E}, \mathcal{E})$ by $\mathrm{K}_{\mathrm{E}}$, \ie,
$\mathrm{K}_{\mathrm{E}} \defeq \{ f : \mathrm{E}\times \mathcal{E} \rightarrow \setOfPositiveReals \}$. Given $K \in \mathrm{K}_{\mathrm{E}}  $, define the   operator $  \mathcal{R} $ as
\begin{align}
  \mathcal{R} K(x) \defeq   \int_{\mathrm{E}} K(x, \differential{y})  \eqstop
  \label{eq:row_sum}
\end{align}
The operator $ \mathcal{R}$ is analogous to the row sum in the finite case, \ie, when $\mathrm{E}$ is finite and $K$ is  a real non-negative matrix. We are interested in getting upper and lower bounds on the spectral radius  $\rho(K)$ of the kernel~$K$. %We call $\rho(K)$ the Perron root or the spectral radius of $K$.  
Before presenting our main result, we have the following two lemmas.

\begin{myLemma}
Suppose $f : \mathrm{E} \rightarrow \setOfPositiveReals$ and $g :\mathrm{E} \rightarrow \setOfPositiveReals$ are such that $\effectiveDomain{f}=\effectiveDomain{g}=\mathrm{E}$ and $ 0 < \int_{\mathrm{E}}  f \differential{ \mu } < \infty,  0 < \int_{\mathrm{E}}  g \differential{ \mu } < \infty $. Then, given a measure~$\mu$ on $\mathrm{E} $,
\begin{align}
  \inf_{x \in \mathrm{E}  } \frac{ f(x) }{  g (x) }  \le \frac{ \int_{\mathrm{E}}  f \differential{ \mu }   }{ \int_{\mathrm{E}}  g \differential{\mu} }   \le   \sup_{x \in \mathrm{E}  } \frac{ f(x) }{  g (x) } \eqstop
\end{align}
\label{lemma:lemma1}
\end{myLemma}

\begin{proof}[Proof of Lemma~\ref{lemma:lemma1}]
Note that
\begin{align*}
 \frac{ \int_{\mathrm{E}}  f \differential{ \mu }   }{ \int_{\mathrm{E}}  g \differential{\mu} }   =  \frac{ 1 }{ \int_{\mathrm{E}}  g \differential{\mu} }   \int_{\mathrm{E}}  \frac{f}{g} g \differential{ \mu }  
 = \int_{\mathrm{E}}  \frac{f}{g}  \differential{ \nu }   \eqcomma
\end{align*}
where the probability measure $\nu$ on $(\mathrm{E}, \mathcal{E})$ is defined by
\begin{align*}
  \nu (A) \defeq \frac{ 1 }{ \int_{\mathrm{E}}  g \differential{\mu} }   \int_{\mathrm{A}}  g \differential{\mu} , \text{ for } A \in \mathcal{E}  \eqstop
\end{align*}
Note that,
\begin{align*}
   \int_{\mathrm{E}}  \frac{f}{g}  \differential{ \nu }   \ge {} & \inf_{x \in \mathrm{E}  } \frac{ f(x) }{  g (x) }  \int_{\mathrm{E}}  \differential{ \nu }  =  \inf_{x \in \mathrm{E}  } \frac{ f(x) }{  g (x) } \nu(\mathrm{E})  = \inf_{x \in \mathrm{E}  } \frac{ f(x) }{  g (x) }  \eqcomma \\
\text{and } \int_{\mathrm{E}}  \frac{f}{g}  \differential{ \nu }   \le {} & \sup_{x \in \mathrm{E}  } \frac{ f(x) }{  g (x) }  \int_{\mathrm{E}}  \differential{ \nu }   =  \sup_{x \in \mathrm{E}  } \frac{ f(x) }{  g (x) } \nu(\mathrm{E}) = \sup_{x \in \mathrm{E}  } \frac{ f(x) }{  g (x) } \eqstop
\end{align*}
Therefore, we get
\begin{align}
  \inf_{x \in \mathrm{E}  } \frac{ f(x) }{  g (x) }  \le \frac{ \int_{\mathrm{E}}  f \differential{ \mu }   }{ \int_{\mathrm{E}}  g \differential{\mu} }   \le   \sup_{x \in \mathrm{E}  } \frac{ f(x) }{  g (x) } \eqstop
\end{align}

%\tdWasiur[]{check if we need any restriction on $\mu$}
\end{proof}

Next, we define the product of two kernels on the measurable space $(\mathrm{E}, \mathcal{E})$. Let $F$, and $ G$ be two kernels on $(\mathrm{E}, \mathcal{E})$. Analogous to matrix multiplication, define the new kernel $F \cdotp G  :   \mathrm{E}\times \mathcal{E} \rightarrow   \setOfReals  \cup \{ \infty, - \infty \} $ as
\begin{align}
 F \cdotp  G    (x, A) \defeq \int_{\mathrm{E}}  F(x, \differential{y} )   G(y, A) \eqstop
\end{align}
The next lemma concerns the row sums of the product of two kernels.

\begin{myLemma}
Let $F$, and $ G$ be two kernels on $(\mathrm{E}, \mathcal{E})$ such that $\effectiveDomain{   \mathcal{R}  F \cdotp  G  }=\mathrm{E}$. And suppose, that $\int_{E} \int_{E} \vert F(x, \,\mathrm d y) G(y, \,\mathrm d z) \vert < \infty$ holds. Then,
\begin{align}
  \mathcal{R}  F \cdotp  G  (x) = \int_{ \mathrm{E}  } F( x, \differential{y} ) \mathcal{R} G(y)  \eqstop
\end{align}
\label{lemma2}
\end{myLemma}

\begin{proof}[Proof of Lemma~\ref{lemma2}]
  Note that by Fubini's theorem,
\begin{align*}
  \mathcal{R}  F \cdotp  G  (x) = {} &  \int_{ \mathrm{E}  }    F \cdotp  G   (x, \differential{z}) \\
  = {} &  \int_{ \mathrm{E}  }  \int_{\mathrm{E}}  F(x, \differential{y} )   G(y,  \differential{z} )  \\
  ={} &  \int_{ \mathrm{E}  }  F(x, \differential{y} )  \int_{\mathrm{E}}    G(y,  \differential{z} )  \\
  ={} &   \int_{ \mathrm{E}  }  F(x, \differential{y} )   \mathcal{R} G(y)   \eqstop
\end{align*}

This completes the proof.

\end{proof}

Now we present our main result on the Perron root $\rho(K)$ of $K$.

\begin{myTheorem}
Let $K$ be a real-valued non-negative kernel, \ie,  $K \in \mathrm{K}_{\mathrm{E}}$.  Let $L$ be a kernel on $(\mathrm{E}, \mathcal{E})$, which is not necessarily in $ \mathrm{K}_{\mathrm{E}}$, but satisfies $\int_{E} \int_{E} \vert K(x, \,\mathrm d y) L(y, \,\mathrm d z) \vert < \infty$, $ \mathcal{R} L (x) >0 $ for all $x \in \mathrm{E}$ and is integrable with respect to the left eigenmeasure of $K$. Then,
\begin{align*}
  \inf_{x \in \mathrm{E}  } \frac{ \mathcal{R} K \cdotp L (x) }{  \mathcal{R} L (x) } \le \rho(K) \le   \sup_{x \in \mathrm{E}  } \frac{ \mathcal{R} K \cdotp L (x) }{  \mathcal{R} L (x) } \eqstop
\end{align*}
\label{thm:mainTheorem}
\end{myTheorem}

\begin{proof}[Proof of Theorem~\ref{thm:mainTheorem}]
Let $r_{K} :\mathrm{E} \rightarrow \setOfReals  $ and $l_{K} : \mathcal{E} \rightarrow \setOfPositiveReals  $ be respectively the right eigenfunction and the left eigenmeasure of $K$ associated with $\rho(K)$. That is,
\begin{align*}
\rho(K)  r_{K}(x)  = {}& \int_{\mathrm{E}} K(x, \differential{y} )   r_{K}(y) \eqcomma   \forall   x \in \mathrm{E}  \eqcomma \\
\rho(K)  l_{K}(A)  ={} &   \int_{\mathrm{E}}   l_K(\differential{x})   K( x, A )  \eqcomma  \forall  A \in \mathcal{E} \eqstop
\end{align*}
Note that $\rho(K)$ is the largest simple eigenvalue of $K$. The existence of $\rho(K),   l_K$, and $r_K$ are guaranteed by  \cite[Theorem III.10.1]{Harris1963Branching}.

 Let $L$ be an arbitrary kernel on $(\mathrm{E}, \mathcal{E})$, which is not necessarily in $ \mathrm{K}_{\mathrm{E}}$, but satisfies $\int_{E} \int_{E} \vert K(x, \,\mathrm d y) L(y, \,\mathrm d z) \vert < \infty$,  $ \mathcal{R} L (x) >0 $ for all $x \in \mathrm{E}$ and is integrable with respect to $l_{K}$ as a measure on $\mathrm{E}$. Then, by the definition of the left eigenmeasure $l_K$, we have
\begin{align*}
    \int_{\mathrm{E}}  l_K(\differential{x})  K( x,  \differential{y}  )   ={} &  \rho(K)  l_{K}(  \differential{y} )  \\
    \implies   \int_{\mathrm{E}}   \mathcal{R} L (y)      \int_{\mathrm{E}}  l_K(\differential{x})  K( x,  \differential{y}  )   ={} &  \rho(K)  \int_{\mathrm{E}}  l_{K}(  \differential{y} )    \mathcal{R} L (y)       \\
    \implies  \int_{\mathrm{E}}  l_K(\differential{x})   \int_{\mathrm{E}}  K( x,  \differential{y}  )   \mathcal{R} L (y)    ={} &  \rho(K)  \int_{\mathrm{E}}  l_{K}(  \differential{y} )   \mathcal{R} L (y) \\
    \implies  \int_{\mathrm{E}}  l_K(\differential{x})     \mathcal{R}  K \cdotp  L  (x)  ={} &  \rho(K)  \int_{\mathrm{E}}  l_{K}(  \differential{y} )   \mathcal{R} L (y) \eqcomma
\end{align*}
because $  \mathcal{R}  K \cdotp  L (x) = \int_{ \mathrm{E}  } K( x, \differential{y} ) \mathcal{R} L(y)  $, by Lemma~\ref{lemma2}. Therefore we have
\begin{align*}
   \rho(K)  ={} & \frac{   \int_{\mathrm{E}}  l_K(\differential{x})     \mathcal{R}  K \cdotp  L  (x)    }{ \int_{\mathrm{E}}  l_{K}(  \differential{x} )   \mathcal{R} L (x) }  \eqstop
\end{align*}
Choosing $\mu = l_K$, $f = \mathcal{R}  K \cdotp  L    $, and $g =   \mathcal{R} L $ in Lemma~\ref{lemma:lemma1}, we get
\begin{align}
  \inf_{x \in \mathrm{E}  } \frac{ \mathcal{R} K \cdotp L (x) }{  \mathcal{R} L (x) } \le \rho(K) \le   \sup_{x \in \mathrm{E}  } \frac{ \mathcal{R} K \cdotp L (x) }{  \mathcal{R} L (x) } \eqstop
\end{align}
This completes the proof.

%\tdWasiur[]{should last inf and sup  be over x where $l_K$ doesn't vanish??}

\end{proof}

The condition $\int_{E} \int_{E} \vert F(x, \,\mathrm d y) G(y, \,\mathrm d z) \vert < \infty$  is necessary and hence, must not be omitted,  as shown in the following example. 

\begin{myExample}
% The following example shows that the condition $\int_{E} \int_{E} \vert F(x, \,\mathrm d y) G(y, \,\mathrm d z) \vert < \infty$ must not be omitted. 
Consider the complete separable metric space $E = [0,1]$. Define $F(x,\,\mathrm d y) = \,\mathrm d y$ to be the Lebesgue measure on $[0,1]$ independent of $x$. The kernel $F$ is non-negative and allows for the Lebesgue measure as left eigenmeasure. Furthermore, define  the function $g \colon E \times E \to \mathbb R$
\[
g(y,z) := \begin{cases}
-\frac{1}{z^2}&, \quad 0 <y < z \le 1\\
\frac{1}{y^2}&, \quad 0 < z \le y \le 1\\
-1&, \quad z = 0, y \neq 0\\
1&, \quad y = 0 \eqstop 
\end{cases}
\]
Note that $g$ is measurable with respect to the Lebesgue measure on $E\times E$ and defines the kernel
\[
G(y,\,\mathrm d z) = g(y,z) \,\mathrm d z.
\]
For all $y \in (0,1]$, calculate
\[
\mathcal R G(y) = \int_0^1 g(y,z) \,\mathrm d z = \int_0^y \frac{1}{y^2} \,\mathrm d z - \int_{y}^1 \frac{1}{z^2} \,\mathrm d z = 1 > 0 \eqcomma 
\]
and for $y = 0$,
\[
\mathcal R G(0) = \int_0^1 \,\mathrm d z = 1 >0.
\]
Further, $\mathcal R G(y)$ is integrable with respect to the left eigenmeasure of $F$.
The product of the kernels is
\[
F\cdot G (x, \,\mathrm d z) = \left( \int_0^1 g(y,z) \,\mathrm d y \right) \,\mathrm d z = \left(\int_0^z -\frac{1}{z^2} \,\mathrm d y + \int_z^1 \frac{1}{y^2} \,\mathrm d y \right) \,\mathrm d z = - \,\mathrm d z.
\]
Hence
\[
\mathcal R F\cdot G(x) = \int_0^1 F\cdot G (x, \,\mathrm d z) = -1
\]
independent of $x$. This implies that $\inf_{x \in \mathrm{E}  } \frac{ \mathcal{R} F \cdotp G (x) }{  \mathcal{R} G(x) }$ and $\sup_{x \in \mathrm{E}  } \frac{ \mathcal{R} F \cdotp G (x) }{  \mathcal{R} G (x) }$ both equal $-1$, even though the Perron root $\rho(F)$ is $1$.

\end{myExample}

% \tdWasiur[]{existence of the Perrot root is guaranteed by \citet[Theorem III.10.1]{Harris1963Branching}   }

%\setcounter{equation}{0}
%\input{Sections/Applications}

%\setcounter{equation}{0}
%\input{Sections/Related_work}
%
%\setcounter{equation}{0}
%\input{Sections/Discussions}

\appendix
% \input{sections/Appendix_A}

%
%\refstepcounter{dummy}
%\pdfbookmark[1]{Acronyms}{acronyms}
%\markboth{\spacedlowsmallcaps{Acronyms}}{\spacedlowsmallcaps{Acronyms}}
%

\section*{Acronyms}

%\addtocontents{toc}{\protect\vspace{\beforebibskip}} 
%\phantomsection\addcontentsline{toc}{chapter}{\tocEntry{Acronyms}}
%%\addtocontents{toc}{\protect\vspace{\beforebibskip}} 

	\acrodefplural{GBP}{General Branching Processes}
\begin{acronym}[OWL-QN]
	\acro{ABM}{Agent-based Model}   
	\acro{ADMM}{Alternating Direction Method of Multipliers} 
	\acro{BA}{Barab\'asi-Albert}
	\acro{BCS}{Bioinspired Communication Systems}
	\acro{CBQA}{Cost-Based Queue-Aware}
	\acro{CCDF}{Complementary Cumulative Distribution Function}  
	\acro{CDF}{Cumulative Distribution Function}   
	\acro{CDN}{Content Distribution Network}   
	\acro{CIM}{Conditional Intensity Matrix}
	\acro{CLT}{Central Limit Theorem}   
	\acro{CM}{Configuration Model}    
	\acro{CME}{Chemical Master Equation}   
	\acro{CRC}{Collaborative Research Centre}   
	\acro{CRN}{Chemical Reaction Network}   
	\acro{CTBN}{Continuous Time Bayesian Network} 
	\acro{CTMC}{Continuous Time Markov Chain}  
	\acro{DCFTP}{Dominated Coupling From The Past }
	\acro{DFG}{German Research Foundation}   
	\acro{DTMC}{Discrete Time Markov Chain}  
	\acro{DTMC}{Discrete Time Markov Chain} 
	\acro{ECMP}{Equal-cost Multi-path routing}   
	\acro{EDF}{Earliest Deadline First}   
	\acro{ER}{Erd\"{o}s-R\'{e}nyi}
	\acro{ESI}{Enzyme-Substrate-Inhibitor}     
	\acro{FCFS}{First Come First Served}  
	\acro{FCLT}{Functional Central Limit Theorem}  
	\acro{FIFO}{First In First Out}  
	\acro{FJ}{Fork-Join}   
	\acro{GBP}{General Branching Process}  

	\acro{ID}{Information-Dissemination}     
	\acro{iid}{independent and identically distributed}  
	\acro{IoT}{Internet of Things}   
	\acro{IPS}{Interacting Particle System}   
	\acro{IT}{Information Technology}    
	\acro{JIQ}{Join-Idle-Queue}     
	\acro{JMC}{Join-Minimum-Cost}
	\acro{JSQ}{Join-Shortest-Queue}     
	\acro{KL}{Kullback-Leibler}     
	\acro{LDF}{Latest Deadline First}   
	\acro{LDP}{Large Deviations Principle}   
	\acro{LLN}{Law of Large Numbers}  
	\acro{LNA}{Linear Noise Approximation}
	\acro{MABM}{Markovian Agent-based Model}   
	\acro{MAKI}{Multi-Mechanism Adaptation for the Future Internet}
	\acro{MAPK}{Mitogen-activated Protein Kinase}
	\acro{MDS}{Maximum Distance Separable}
	\acro{MGF}{Moment Generating Function}   
	\acro{MM}{Michaelis-Menten}     
	\acro{MPI}{Message Passing Interface}   
	\acro{MPTCP}[Multi-path TCP]{Multi-path Transmission Control Protocol}   
	\acro{ODE}{Ordinary Differential Equation}   
	\acro{P2P}{Peer-to-Peer}     
	
	\acro{PDF}{Probability Density Function}
	\acro{PGF}{Probability Generating Function}   
	\acro{PGM}{Probabilistic Graphical Model}
	\acro{PMF}{Probability Mass Function}
	\acro{psd}{positive semi-definite}
	\acro{PT}{Poisson-type}
%	\acro{QSSA}{quasi-steady state approximation}   
	\acro{QSSA}{Quasi-Steady State Approximation}   
	\acro{rQSSA}{reversible QSSA}    
	\acro{SAN}{Stochastic Automata Network}
	\acro{SEIR}{Susceptible-Exposed-Infected-Recovered}     
	\acro{SI}{Susceptible-Infected}     
	\acro{SIR}{Susceptible-Infected-Recovered}     
	\acro{SIS}{Susceptible-Infected-Susceptible}     
	\acro{sQSSA}{standard QSSA}    
	\acro{SRPT}{Shortest Remaining Processing Time}  
	\acro{ssLNA}{Slow-scale Linear Noise Approximation}
	\acro{TCP}{Transmission Control Protocol}   
	\acro{tQSSA}{total QSSA}  
	\acro{WS}{Watts-Strogatz}  	
	\acro{whp}{with high probability}
\end{acronym}

%\vfill 

\section*{ACKNOWLEDGEMENTS}
This work  has been funded by the German Research Foundation~(DFG) as part of project C3 within the Collaborative Research Center~(CRC) 1053 -- MAKI.

\bibliographystyle{abbrv}

%\bibliography{Bib_WKB,Bib_HK}

\end{document}